\documentclass[]{amsart}
 \usepackage[]{amsmath, amsthm, amsfonts, amssymb}
 \usepackage[all]{xy}
\usepackage[mathscr]{eucal}
 \usepackage{hyperref}
\xyoption{curve}
 
 \newcommand {\C} {{\mathbb C}}
 
 \newcommand {\Z} {{\mathbb Z}}
 \newcommand {\Q} {{\mathbb Q}}

 \newcommand {\Ss} {{\mathcal S}}
  \newcommand {\Y} {{\mathcal Y}}
  
 \newcommand {\K} {{\mathcal K}}
  
  \newcommand {\M} {{\mathcal M}}

 \newcommand {\X} {{\mathscr X}}

 \newcommand {\coker}{\text{\rm coker}}

 \newcommand {\cl}{\text{\rm cl}}
 \newcommand {\ch}{\text{\CH}}

\newcommand{\Spec}{\text{\rm Spec}}
\newcommand{\CH}{\text{\rm CH}}

\newcommand{\ol}{\overline}
\newcommand{\ul}{\underline}

\newcommand{\IM}{\text{\rm Image}}

\newcommand{\Ext}{\text{\rm Ext}}
\newcommand{\MHS}{\text{\rm MHS}}

\newcommand{\bs}{{\backslash}}

 \newtheorem{thm}[equation]{Theorem}
 \newtheorem{cor}[equation]{Corollary}
 \newtheorem{lemma}[equation]{Lemma}
 \newtheorem{prop}[equation]{Proposition}
 \newtheorem{conj}[equation]{Conjecture}
 
 \newtheorem{rmk}[equation]{Remark}
 \newtheorem{ex}[equation]{Example}
 
 \newtheorem{q}[equation]{Question}

 \theoremstyle{definition}
 \newtheorem{ass}[equation]{Assumptions}
 
 \numberwithin{equation}{section} 

 \usepackage[usenames,dvipsnames]{color}

\begin{document}

\title[Beilinson's Hodge Conjecture]{A relative version of the Beilinson-Hodge conjecture}

\author{Rob de Jeu}

\address{Faculteit Exacte Wetenschappen\\  Afdeling Wiskunde \\
VU University Amsterdam \\ The Netherlands}

\email{r.m.h.de.jeu@vu.nl}

\author{James D. Lewis}

\address{632 Central Academic Building\\
University of Alberta\\
Edmonton, Alberta T6G 2G1, CANADA}

\email{lewisjd@ualberta.ca}

\author{Deepam Patel}

\address{Faculteit Exacte Wetenschappen\\  Afdeling Wiskunde \\
VU University Amsterdam \\ The Netherlands}

\email{deeppatel1981@gmail.com}

\thanks{The authors are grateful to the Netherlands
Organisation for Scientific Research (NWO) 
The second author also acknowledges partial {support} by a grant from the Natural Sciences
and Engineering Research Council of Canada.}



\renewcommand{\abstractname}{Abstract} 

\begin{abstract}  Let $k \subseteq \C$ be an algebraically closed subfield, and $\X$ a variety defined
over $k$. One version of the Beilinson-Hodge conjecture that seems to survive scrutiny is the statement
that the Betti cycle class map $\cl_{r,m} : H_{\M}^{2r-m}(k(\X),\Q(r)) \to \hom_{\MHS}\big(\Q(0),H^{2r-m}(k(\X)(\C),\Q(r))\big)$
is surjective, that being equivalent to the Hodge conjecture in the case $m=0$. Now consider
a smooth and proper map $\rho : \X \to \Ss$ of smooth quasi-projective varieties over $k$, and where $\eta$ is 
the generic point of $\Ss$. We anticipate that the corresponding cycle class map is surjective, and provide some
evidence in support of this in the case where $\X = \Ss\times X$ is a product and $m=1$.
\end{abstract}

\maketitle

\section{Introduction}\label{S1}

The results of this paper are aimed at providing some evidence in support of an
affirmative answer to a question first formulated in~\cite[Question~1.1]{SJK-L},
now upgraded to the following:

\begin{conj}\label{MC}
Let $\rho : \X \to \Ss$ be a smooth proper map of smooth quasi-projective varieties over a
subfield $k = \ol{k}\subseteq \C$, with $\eta= \eta_{\Ss}$ the generic point of $\Ss/k$. Further, let $r,m \geq 0$
be integers.  Then
\[
\cl_{r,m} : \CH^r(\X_{\eta},m;\Q) = H_{\M}^{2r-m}(\X_{\eta},\Q(r)) \to \hom_{\MHS}\big(\Q(0),H^{2r-m}(\X_{\eta}(\C),\Q(r))\big),
\]
is surjective.
\end{conj}

Here 
\[
H^{2r-m}(\X_{\eta}(\C),\Q(r)) := \lim_{{\buildrel\to\over{U\subset \Ss/k}}}H^{2r-m}(\rho^{-1}(U)(\C),\Q(r)),
\]
is a limit of mixed Hodge structures (MHS), for which one should not expect finite
dimensionality, and for any smooth quasi-projective variety  $W/k$, we identify 
motivic cohomology $H^{2r-m}_{\M}(W,\Q(r))$ with Bloch's higher Chow group
$\CH^r(W,m;\Q) := \CH^r(W,m)\otimes \Q$ (see~\cite{Bl1}). Note that
if $\Ss = \Spec(k)$, and $m=0$, then $\X = X_k$ is smooth, projective over $k$.
Thus in this case Conjecture~\ref{MC}  reduces to the (classical) Hodge
conjecture.
The motivation for this conjecture stems from the following: 

\medskip
Firstly, it
is a generalization of similar a conjecture in~\cite[(\S1, statement (S3)]{dJ-L}, where $\X = \Ss$,
based on a generalization of the Hodge conjecture (classical form) to the
higher $K$-groups, and inspired in part by Beilinson's work in this direction.

\medskip
In passing,  we hope to instill in the reader  that any attempt to deduce Conjecture~\ref{MC}
from~\cite[\S1, statement (S3)]{dJ-L} seems to be hopelessly naive, and would
require some new technology. To move ahead with this, we eventually work in
the special situation where $\X = S\times X$ is a product, with $S = \Ss$ and
$X$ smooth projective, $m=1$, and employ some 
motivic input, based on reasonable pre-existing conjectures.

\medskip
Secondly, as a formal application of M. Saito's theory of mixed Hodge modules 
(see~\cite{A},~\cite{K-L}, ~\cite{SJK-L} and the references cited
there),  one could conceive of the following short exact sequence:
\begin{equation} \label{GM}
\begin{matrix} 0\\
&\\
\biggl\downarrow\\
&\\
{\Ext}^1_{\text{\rm P}\MHS}\big(\Q(0),H^{\nu-1}
(\eta_{\Ss},R^{2r-\nu-m}\rho_{\ast}\Q(r))\big)\\
\matrix {}_{\text{\rm Graded\ polar-}}\\
{}^{\text{\rm izable\ MHS}}\endmatrix^{\nearrow}\hskip2.6in\\
\biggl\downarrow\\
&\\
\left\{\begin{matrix} \text{\rm Germs\ of\ higher\ order}\\
\text{\rm generalized\ normal\ functions}
\end{matrix}\right\}\\
&\\
\biggl\downarrow\\
&\\
\hom_{\MHS}\big(\Q(0),H^{\nu}({\eta_{\Ss}},R^{2r-m-\nu}\rho_{\ast}\Q(r))\big)\\
&\\
\biggr\downarrow\\
&\\
0\end{matrix}
\end{equation}
(Warning: As mentioned earlier, passing to the generic point $\eta_{\Ss}$ of $\Ss$ is a limit process, which implies that
the spaces above need not be finite dimensional over $\Q$. This particularly applies to the case $m\geq 1$,
where there are residues.)
The key point is, is there lurking a generalized Poincar\'e
existence theorem for higher normal functions?  Namely, modulo the ``fixed part'' 
${\Ext}^1_{\text{\rm P}\MHS}\big(\Q(0),H^{\nu-1}
(\eta_{\Ss},R^{2r-\nu-m}\rho_{\ast}\Q(r))\big)$, are
these normal functions cycle-induced? 
In another direction, this diagram is related to
a geometric description of the
notion of a Bloch-Beilinson (BB)  filtration.
As a service to the reader, and to make sense of this all,
we elaborate on all  of this.

\medskip
\noindent
1. For the moment, let us replace $(\eta_{\Ss}$ by $\Ss$, $(\nu,m)$ by $(1,0)$ in diagram (\ref{GM}), 
and where $\Ss$ is  chosen to be a curve). Then
this diagram represents the schema of the original Griffiths program aimed at generalizing Lefschetz's famous $(1,1)$ theorem,
via normal functions.\footnote{Technically speaking, Griffiths worked with normal functions that extended to
the boundary $\ol{\Ss}\bs \Ss$, but let's not go there.} This program was aimed at solving the
Hodge conjecture inductively. Unfortunately, the lack of a Jacobi inversion theorem for the jacobian
of a general smooth projective variety involving a Hodge structure of weight $> 1$ led to limited
applications towards the Hodge conjecture. However the qualitative aspects of his program led
to the non-triviality of the now regarded Griffiths group. In that regard, the aforementioned diagram represents a 
generalization of this idea to the higher $K$-groups of $\X$ and the general fibers of $\rho :\X\to \Ss$.

\medskip
\noindent
2.  The notion of a BB filtration, first suggested by Bloch and later fortified by Beilinson, tells us 
that for any $X/k$ smooth projective and $r, m\geq 0$, there should be a descending filtration
\[
\big\{F^{\nu}\CH^r(X,m;\Q) \ \big|\ \nu  = 0,...,r\big\},
\]
whose graded pieces can be described in terms of extension datum, viz.,
\[
Gr_F^{\nu}\CH^r(X,m;\Q) \simeq \Ext^{\nu}_{\M\M}(\Spec(k),h^{2r-\nu-m}(X)(r)),
\]
where $\M\M$ is a conjectural category of mixed motives and $h^{\bullet}(X)(\bullet)$ is
motivic cohomology.\footnote{The original formulation involved
only the case $m=0$; this is just a natural extension of these ideas to the higher $K$-groups of $X$.}
Although there were many excellent candidate BB filtrations proposed by others over
the years,  a few are derived from the point of view of  ``spreads'',
in the case $k = \ol{\Q}$ (see~\cite{A},~\cite{Lew1}, ~\cite{GG})
as well as a conjectural description in terms of normal functions
(see~\cite{K-L}, \cite{Lew3}). Namely, if $X/\C$ is smooth and projective, then there
is a field $K$ of finite transcendence degree over $\ol{\Q}$ and a smooth and proper
spread $\X \xrightarrow{\rho} \Ss$ of smooth quasi-projective varieties over $\ol{\Q}$,
such that if $\eta$ is the generic point of $\Ss$, then $K$ can be identified with
$\ol{\Q}(\eta)$ via a suitable embedding $\ol{\Q}(\eta)\hookrightarrow \C$; moreover
with  respect to that embedding, $X/\C = \X_{\eta}\times_{\ol{\Q}(\eta)}\C$. Diagram
(\ref{GM}) then provides yet another schema of describing a candidate BB filtration
in terms of normal functions.

\medskip
As indicated  earlier, we focus our attention mainly on the case $m=1$ ($K_1$ case),
and provide some partial results in the case where $\X = S\times X$ is a product, with 
$S = \Ss$, and $X$ smooth projective. Our main results are Theorems~\ref{MT0},~\ref{MT1} and~\ref{MT2}.

\section{Notation} \label{NT}

\noindent
$\bullet_0$  Unless specified to the contrary, all varieties are defined over $\C$.

\medskip
\noindent
$\bullet_1$ $\Q(m)$ is the Tate twist with Hodge type $(-m,-m)$.

\medskip
\noindent
$\bullet_2$ For a mixed Hodge structure (MHS) $H$ over $\Q$, we put
$\Gamma(H) = \hom_{\MHS}(\Q(0),H)$ and 
$J(H) = \Ext^1_{\MHS}(\Q(0),H)$.

\medskip
\noindent
$\bullet_3$ The higher Chow groups $\CH^r(W,m)$ for a quasi-projective variety $W$ over a field $k$
are defined in~\cite{Bl1}. Let us assume $W/k$ is regular. An abridged definition 
of $\CH^r(W,1)$, viz., in the case $m=1$ is given by:
\[
\CH^r(W,1) = \frac{\ker \big( \sum_{{\rm cd}_WZ=r-1}^{Z\ {\rm irred}}(k(Z)^{\times},Z)\xrightarrow{\rm div}  z^r(W)\big)}{{\rm Image\ of\ tame\ symbol}},
\]
where $z^r(W)$ is the free abelian group generated by (irreducible) subvarieties of codimension $r$ in $W$;
moreover, the denominator admits this description.
 If $V\subset W$ is an irreducible subvariety of codimension $r-2$, and
$f,g \in k(V)^{\times}$, then the tame symbol is given as:
\[
T(\{f,g\}_V) = \sum_{{\rm cd}_VD=1}(-1)^{\nu_D(f)\nu_D(g)}\biggl(\frac{f^{\nu_D(g)}}{g^{\nu_D}(f)}\biggr)_D,
\]
as $D$ ranges through all irreducible codimension one subvarieties of $V$,  $\nu_D(\cdot )$ is the
order of a zero or pole, and $\big(\cdots \big)_D$ means the restriction to the generic point of $D$.
The ``Image of the tame symbol'' is the subgroup generated by $T(\{f,g\}_V)$, as $V$ ranges in $W$ and
$f,g$ range through $k(V)^{\times}$.

\medskip
\noindent
$\bullet_4$ Assume $W$ in $\bullet_3$ is also smooth of dimension $d_W$, and let $Z\subset W$ be irreducible of
codimension $r-1$, $f\in k(W)^{\times})$, where $k\subset \C$ is a subfield. Then the
Betti class  map 
\[
\cl_{r,1}: \CH^r(W,1;\Q) \to \Gamma\big(H^{2r-1}(W,\Q(r))\big)\subset H^{2r-1}(W,\Q(r)) \simeq
\]
\[
 \big[H^{2d_W-2r+1}_c(W,\Q(d_W-r))\big]^{\vee} \subset  \big[H^{2d_W-2r+1}_c(W,\C)\big]^{\vee}
\]
is induced by the current
\[
(Z,f) \mapsto \frac{1}{(2\pi{\rm {i})^{d_w-r+1}}}\int_Zd\log(f)\wedge \omega,\quad \{\omega\}\in H^{2d_W-2r+1}_c(W,\C).
\]

\section{What is known}\label{SN}

In this section, we summarize some of the results in~\cite{SJK-L}, where
$r\geq m = 1$.
The setting is the following diagram
\begin{equation*}
\xymatrix{
\X \ar@{^(->}[r] \ar[d]^{\rho} & \ol{\X} \ar[d]^{\ol{\rho}}\\
\Ss \ar@{^(->}[r] & \ol{\Ss}
}
\end{equation*}
where $\ol{\X}$ and $\ol{\Ss}$ are nonsingular complex 
projective varieties, $\ol{\rho}$ is a dominating flat morphism,
$D\subset \ol{\Ss}$ a divisor, $\Y:= \ol{\rho}^{-1}(D)$,
$\Ss := \ol{\Ss}\bs D$, $\X := \ol{\X}\bs \Y$ and $\rho := \ol{\rho}\big|_{\X}$.\footnote{While we assume that the base field is $\C$, the results here
are valid for varieties over an algebraically closed field $k \subset \C$.}
There is a short exact sequence
\begin{equation} \label{eq:2.2}
 0\to \frac{H^{2r-1}(\ol{\X},\Q(r))}{H_{\Y}^{2r-1}(\ol{\X},\Q(r))} 
 \to H^{2r-1}(\X,\Q(r)) 
 \to  H^{2r}_{\Y}(\ol{\X},\Q(r))^{\circ}\to 0,
\end{equation}
where with regard to the former term in~\eqref{eq:2.2}, 
$H_{\Y}^{2r-1}(\ol{\X},\Q(r))$ is identified with its image
in $H^{2r-1}(\ol{\X},\Q(r))$, and $H^{2r}_{\Y}(\ol{\X},\Q(r))^{\circ} := \ker \big(H^{2r}_{\Y}(\ol{\X},\Q(r))\to 
H^{2r}(\ol{\X},\Q(r))\big)$.
One has a corresponding  diagram 
\begin{equation} \label{eq:2.3}
\begin{split}
\xymatrix{
\ch^{r}(\X,1;\Q)\ar[r]\ar[d]^{\cl^{\X}_{r,1}}&\ch^{r}_{\Y}(\ol{\X};\Q)^{\circ}
\ar[r]^{\alpha_{\Y}} \ar[d]^{\beta_{\Y}}& \ch^{r}_{\hom}(\ol{\X};\Q) \ar[d]^{\ul{AJ}^{\ol{\X}}}\\
\Gamma\big( H^{2r-1}(\X,\Q(r))\big) \ar@{^(->}[r]&\Gamma\big(
H^{2r}_{\Y}(\ol{\X},\Q(r))^{\circ}\big) \ar[r]&J\biggl(\frac{H^{2r-1}(\ol{\X},\Q(r))}{H_{\Y}^{2r-1}(\ol{\X},\Q(r))}\biggr)
}
\end{split}
\end{equation}
well-known to commute
by an extension class argument, and where $\ul{AJ}^{\ol{\X}}$ is the corresponding ``reduced'' Abel-Jacobi map.
Further, the definition of  $\ch^{r}_{\Y}(\ol{\X};\Q)^{\circ}$ is the obvious one, being the cycles
in $\CH^{r-1}(\Y;\Q)$ that are homologous to zero on $\ol{\X}$. 

\begin{rmk} \label{KRK}
Poincar\'e duality gives an isomorphism of MHS:
\[
H^{2r}_{\Y}(\ol{\X},\Q(r)) \simeq H_{2\dim\ol{\X} -2r}({\Y},\Q(\dim\ol{\X} -r)).
\]
Thus $\ker \beta_{\Y} = \CH^{r-1}_{\hom}(\Y;\Q)$, viz.,  the subspace of cycles in $\CH^{r-1}(\Y;\Q)$
that are homologous to zero on $\Y$.
\end{rmk}

Let us
assume that $\beta_{\Y}$ is surjective, as is the case 
if the (classical) Hodge conjecture holds. 
If we apply the snake lemma, we arrive at 
\begin{equation*}
\coker(\cl^{\X}_{r,1}) \simeq 
\frac{\ker\biggl[ \ul{AJ}^{\ol{\X}}\big|_{\IM(\alpha_{\Y})}:
\IM(\alpha_{\Y}) \to J\biggl(\frac{H^{2r-1}(
\ol{\X},\Q(r))}{H_{\Y}^{2r-1}(\ol{\X},\Q(r))}\biggr)
\biggr]}{\alpha_{\Y}\big(\ker(\beta_{\Y})\big)}. 
\end{equation*}
Now take the limit over all $D\subset \ol{\Ss}$ to arrive at an
induced cycle map:
\begin{equation} \label{eq:2.4.1}
  \cl^{\eta}_{r,1} : \CH^{r}(\X_{\eta},1;\Q) \to \Gamma\big(H^{2r-1}(\X_{\eta},\Q(r))\big). 
\end{equation}
where $\eta$ is the generic point of $\ol{\Ss}$. We arrive at:
\begin{equation} \label{eq:2.5}
  \frac{\Gamma\big(H^{2r-1}(\X_{\eta},\Q(r))\big)}{\cl^{\eta}_{r,1}\big(
  \CH^{r}(\X_{\eta},1;\Q)\big)} \simeq\
  \frac{\ker\biggl[\K
  \stackrel{\ul{AJ}}{\to} J\biggl(
  \frac{H^{2r-1}(\ol{\X},\Q(r))}{N_{\ol{\Ss}}^{1}H^{2r-1}(
  \ol{\X},\Q(r))}\biggr)\biggr]}{N_{\ol{\Ss}}^{1}\CH^{r}(\ol{\X};\Q)},
\end{equation}
where $\K := \ker [\CH^{r}_{\hom}(\ol{\X};\Q)\to
\CH^{r}(\X_{\eta};\Q)]$,
$N_{\ol{\Ss}}^{q}\CH^{r}(\ol{\X}) \subset 
\CH_{\hom}^{r}(\ol{\X};\Q)$ is the subgroup
generated by cycles which are homologous to zero on some
codimension $q$ subscheme of $\ol{\X}$ obtained from a  (pure) codimension
$q$ subscheme of $\ol{S}$ via $\ol{\rho}^{-1}$ (keep in mind Remark~\ref{KRK}, in the
case $q=1$),
and $N_{\ol{S}}^{q}H^{2r-1}(\ol{\X},\Q(r))$ is the subspace of the
coniveau $N^qH^{2r-1}(\ol{\X},\Q(r))$ arising from $q$
codimensional subschemes
of $\ol{S}$ via $\ol{\rho}^{-1}$. A relatively simple argument, found in~\cite{SJK-L}, yields
the following:

\begin{prop} \label{KeyP}
Under the assumption of the Hodge conjecture,  \eqref{eq:2.5} becomes:
\[
\frac{\Gamma\big(H^{2r-1}(\X_{\eta},\Q(r))\big)}{\cl^{\eta}_{r,1}\big(
\CH^{r}(\X_{\eta},1;\Q)\big)} \simeq\
\frac{N_{\ol{\Ss}}^{1}\CH^{r}(\ol{\X};\Q)\ +\ \ker\big[\K \stackrel{AJ}{\to} J\big(
H^{2r-1}(\ol{\X},\Q(r))\big)\big]}{N_{\ol{\Ss}}^{1}\CH^{r}(\ol{\X};\Q)}.
\]
\end{prop}

\begin{ex}\label{example:2.7} 
Suppose that $\ol{\X} = \ol{\Ss}$ with $\ol{\rho}$ 
the identity. In this case Proposition~\ref{KeyP} becomes:
\begin{equation}\label{eq:2.7}
\frac{\Gamma\big(H^{2r-1}(\C(\ol{\X}),\Q(r))\big)}{\cl_{r,1}\big(
\CH^{r}(\Spec(\C(\ol{\X})),1;\Q)\big)}\  \simeq\
\frac{N^{1}\CH^{r}(\ol{\X};\Q)\ +\  
\CH_{AJ}^r(\ol{\X};\Q)}{N^{1}\CH^{r}(\ol{\X};\Q)},
\end{equation}
where $N^1\CH^r(\ol{\X};\Q)$ is the subgroup of cycles, that are homologous to
zero on codimension $1$ subschemes of $\ol{\X}$, and $\CH_{AJ}^r(\ol{\X};\Q)$ are
cycles in the kernel of the Abel-Jacobi map $AJ : \CH^r(\ol{\X};\Q)
\to J\big(H^{2r-1}(\ol{\X},\Q(r))\big)$. According to Jannsen~\cite[p. 227]{Ja1},
there is a discussion that strongly hints  that the right
hand side of \eqref{eq:2.7} should be  zero. In light of~\cite{Lew2}, we conjecturally believe this to
be true. In particular, 
since $\Spec(\C(\ol{\X}))$ is
a point, this implies that $\Gamma\big(H^{2r-1}(\C(\ol{\X}),\Q(r))\big) 
= 0$ for $r>1$. The reader can easily check that 
$$
\cl_{r,1}\big(\CH^{r}(\Spec(\C(\ol{\X})),1;\Q)\big) = 
\Gamma\big(H^{2r-1}(\C(\ol{\X}),\Q(r))\big),
$$
holds unconditionally in the case $r=\dim \ol{\X}$, that  being well known in
the case $r = \dim \X = 1$, and for $r = \dim \X > 1$, from the weak Lefschetz theorem
for affine varieties.
\end{ex}

\begin{ex}\cite{SJK-L} Here we give some evidence that
the RHS (hence LHS) of Proposition~\ref{KeyP} is zero. 
Suppose $\ol{\X} = X\times \ol{S}$,  ($\ol{S} := \ol{\Ss}$), and let us
assume the condition 
$$
\CH^{r}(\ol{\X};\Q) = \bigoplus_{\ell=0}^r\ \CH^{\ell}(\ol{S};\Q) \otimes\CH^{r-\ell}(X;\Q).
$$
An example
situation is when $\ol{S}$ is a flag variety, 
such as a projective space; however conjecturally
speaking, this condition is expected to hold for
a much broader class of examples. If we further assume the Hodge conjecture,
then as a consequence of Proposition~\ref{KeyP},
we arrive at:

\begin{cor} Under the assumptions on the K\"unneth condition above and the Hodge conjecture, with 
$\ol{\X} = \ol{S}\times X$,  if
$$
\cl_{r,1}\big(\CH^{r}(\Spec(\C(\ol{S})),1;\Q)\big) = 
\Gamma\big(H^{2r-1}(\C(\ol{S}),\Q(r))\big),
$$
holds, then the map 
$\cl_{r,1}^{\eta} $ in \eqref{eq:2.4.1} is surjective.
\end{cor}
\end{ex}

\section{The split case and rigidity}\label{S2}

\subsection{Base a curve}
In this section, we observe that the Beilinson-Hodge conjecture  (Conjecture~\ref{MC}), in the special case of a split projection with base given by a curve, holds under the assumption of the Hodge conjecture on the fibre. Let $X$ be a smooth projective variety and $C$ a smooth curve. Let $\pi: C \times X \rightarrow C$ denote the projection morphism.

\begin{prop}\label{splitcase1}
Let $X$ and $C$ be as above. 

\begin{enumerate}
\item
 If $m > 1$, then $\CH^{r}(C \times X, m) \rightarrow \Gamma(H^{2r-m}(C \times X, \Q(r)))$
is surjective.

\item
If $m =1$, then $\CH^{r}(C \times X, m) \rightarrow \Gamma(H^{2r-m}(C \times X, \Q(r)))$ is surjective
if the Hodge conjecture holds for $X$ in codimension $r-1$.
\end{enumerate}

In particular, the Beilinson Hodge conjecture (Conjecture~\ref{MC}) for $\pi: C \times X \rightarrow C$
holds unconditionally if $m >  1$ and, if $m = 1$, then it holds under the assumption of the Hodge conjecture for $X$. 
 \end{prop}

We begin with some preliminary reductions.  By the K\"unneth decomposition we can identify $ H^{2r-m}(C\times X,\Q) $ with
\begin{equation*}
 \bigoplus_{i=0}^2 H^i(C,\Q)\otimes H^{2r-m-i}(X,\Q)
\,.
\end{equation*}
For $ i = 0 $ and 2, $ H^i(C,\Q)\otimes H^{2r-m-i}(X,\Q)(r) $
is pure of weight $ -m $ by the purity of 
$ H^{2r-m-i}(X,\Q) $ and that of $ H^0(C,\Q) $ as well as $ H^2(C,\Q) $.
Therefore $ \Gamma(H^i(C,\Q)\otimes H^{2r-m-i}(X,\Q)(r) ) = 0 $ for $m > 0$
and $ i \ne 1 $. The same also holds for $ i=1 $ if $ C $ is projective.
But in general we have the following.

\begin{lemma}\label{SE}
With notation and assumptions as above,
\begin{equation*} 
\Gamma \big(H^1(C,\Q)\otimes H^{2r-2}(X,\Q)(r)\big) =  \Gamma (H^1(C,\Q(1))) \otimes \Gamma(H^{2r-2}(X,\Q(r-1))).
\end{equation*}
\end{lemma}

\begin{proof}
Let  $H$ denote $H^{2r-2}(X,\Q(r-1)) $. Setting $W_j = W_jH^1(C,\Q(1))$
 gives the following commutative diagram of mixed Hodge structures with exact rows:
$$
\xymatrix{
0 \ar[r]  & W_{-1}\otimes H \ar[r] &  W_0 \otimes H  \ar[r] & Gr_0^W\otimes H  \ar[r] & 0 \\
0 \ar[r]  & W_{-1}\otimes \Gamma(H) \ar[r] \ar[u] &  W_0 \otimes \Gamma(H)  \ar[r] \ar[u]& Gr_0^W\otimes \Gamma(H)  \ar[r] \ar[u]& 0 
}
$$
Since $W_{-1}\otimes H $ and $W_{-1}\otimes \Gamma(H)$ have negative weights, their
$ \Gamma $'s are trivial. It follows that we have a commutative diagram with exact rows:
$$
\xymatrix{
0 \ar[r]  &  \Gamma(W_0 \otimes H)  \ar[r] & \Gamma(Gr_0^W\otimes H)  \ar[r] & J(W_{-1}\otimes H)  \\
0 \ar[r]  &  \Gamma(W_0 \otimes \Gamma(H))  \ar[r] \ar[u]& \Gamma(Gr_0^W\otimes \Gamma(H))  \ar[r] \ar[u] & J(W_{-1}\otimes \Gamma(H))  \ar[u] 
}
$$
Since $Gr_0^W$ is pure Tate of weight~0, 
the middle vertical is an isomorphism.
We also have an injection $ H^1(\ol{C},\Q) \to H^1(C,\Q) $,
identifying $ W_{-1} $ with $ H^1(\ol{C}, \Q(1)) $.
Using semi-simplicity of polarized Hodge structures we see that
the natural injection
\[
 W_{-1} \otimes \Gamma\left( H \right) \to W_{-1} \otimes H
\]
is split, so that we obtain a split injection
\[
   J\left( W_{-1} \otimes \Gamma(H) \right)
\to
   J\left( W_{-1} \otimes H \right)
\,.
\]
Finally, noting that $\Gamma(W_0 \otimes \Gamma(H))  = \Gamma(W_0 )\otimes \Gamma(H) $ and using the snake lemma gives the desired result.
\end{proof}

\begin{proof}[Proof of Proposition~\ref{splitcase1}]
If $m > 1$, then by a weight argument
$$\Gamma(H^{1}(C, \mathbb{Q}) \otimes H^{2r-m-1}(X, \mathbb{Q})(r)) = 0$$
and the surjectivity is trivial. 
For $r$ arbitrary, and $m=1$, first note that
$ \CH^1(C,1)$  maps surjectively to $\Gamma(H^{1}(C,\Q(1)))$.
Under the assumption of the Hodge conjecture for $ X $ in codimension $r-1$, one also has that
$$ \CH^{r-1}(X,0;\Q) \to \Gamma\left(H^{2r-2}(X, \Q(r-1)) \right) $$ is surjective. Since the natural morphism
$$ \CH^1(C, 1) \otimes \CH^{r-1}(X, 0) \to \CH^{r}(C \times X, 1) $$
induced by pullback to $ C \times X  $ followed by the cup product is compatible with the tensor product in the K\"unneth decomposition, the
previous remarks and Lemma~\ref{SE} show that $\cl_{1,r}$ surjects onto $\Gamma( H^1(C,\Q) \otimes H^{2r-2}(X, \Q) (r-1)) $.
The last claim in the proposition follows from the first two by taking limits over open $U \subset C$.

\end{proof}

\begin{rmk}
Note that in the situation above, the surjectivity in Conjecture~\ref{MC} holds for every open $U \subset C$, and, in particular, one does not need to pass to the generic point. However, in general,
in the examples of~\cite[Section~5]{dJ-L} one has to.
\end{rmk}

\subsection{Base a product of two curves} \label{twocurve}

In this section, we prove the strong form (i.e., without passing to the generic point) of the Beilinson-Hodge conjecture  (Conjecture~\ref{MC}) for $r=2$ and $m=1$ in the special case of a split projection with base given by a product of two
curves, under a certain rigidity assumption (see below). More precisely, let $X$ be smooth projective, and $ \ol{C}_1$, $ \ol{C}_2$
smooth projective curves, with non-empty open $C_j \subsetneq \ol{C}_j$. Let
$\ol{S} = \ol{C}_1\times \ol{C}_2$, $S = C_1\times C_2$,
$\Sigma_j = \ol{C}_j\bs C_j$, and
$E = \ol{S}\bs S = \Sigma_1\times \ol{C}_2 \cup \ol{C}_1\times \Sigma_2$. Finally, let $\X = S \times X$ and
let $\pi: \X \rightarrow S$ denote the canonical projection map. 

\begin{thm}\label{MT0}
Let $X$ and $S$ be as above. If $H^2(\ol{S}, \Q) $ does not have
a non-zero $ \Q $ subHodge structure contained in $H^{2,0}(\ol{S}) \oplus H^{0,2}(\ol{S})$,
then
\[
\cl_{2,1}: \CH^2(S\times X, 1;\Q) \to \Gamma\big(H^3(S\times X,\Q(2))\big)
\]
is surjective.
\end{thm}

Note that if we were to replace $ S $ with $ C_1 \times \ol{C}_2 $ or $ S = C_1 \times \ol{C}_2 $
then the result is already part of Proposition~\ref{splitcase1}.
Again this holds even if the $ C_i $ are complete by reduction
to the case of one curve in the base.

We begin with some preliminary reductions. First observe that, as $\dim S = 2$,
by the weak Lefschetz theorem for affine varieties,
the K\"unneth decomposition of $ H^3(S\times X,\Q) $ is
\begin{equation*} 
H^0(S,\Q)\otimes H^3(X,\Q) \oplus H^1(S,\Q)\otimes H^2(X,\Q) \oplus 
H^2(S,\Q)\otimes H^1(X,\Q)
\,,
\end{equation*}
and we shall deal with the three summands separately
(after twisting with $ \Q(2) $).\\

For the first term, note that $ H^0(S,\Q)\otimes H^3(X,\Q)(2) $
is pure of weight $ -1 $ by the purity of $ H^3(X,\Q) $, so that
$ \Gamma(H^0(S,\Q)\otimes H^3(X,\Q)(2)) = 0 $.\\

\begin{lemma}
With notation as above, 
\begin{equation*}
\Gamma \big(H^1(S,\Q)\otimes H^{2}(X,\Q)(2)\big) =  \Gamma( H^1(S,\Q(1))) \otimes \Gamma(H^{2}(X,\Q(1))
\end{equation*}
and $ \cl_{1,2} $ is surjective.
\end{lemma}

\begin{proof}
The proof of the equality is the same as the proof of  Lemma~\ref{SE}.
Furthermore, the proof of surjectivity of $\cl_{1,2}$ is similar to the proof of Proposition~\ref{splitcase1}. We leave the details to the reader.
\end{proof}

\medskip

For the proof of Theorem~\ref{MT0} it remains to show that the
image of $\cl_{1,2}$ contains $ \Gamma(H^2(S,\Q)\otimes H^1(X,\Q)(2))$.
The following lemma gives a description of the latter in terms of the Abel-Jacobi map.
First  put 
\[
H^3_E(\ol{S},\Q)^{\circ} = \ker \big[H^3_E(\ol{S},\Q) \to H^3(\ol{S},\Q)\big].
\]

\begin{lemma}\label{E2lemma}
One has these identifications:

\medskip
\noindent
{\rm (i)} $ \Gamma(H^2(S,\Q)\otimes H^1(X,\Q)(2)\ \simeq$
\begin{equation*}
\ker \big[\Gamma\big(H^3_E(\ol{S},\Q(2))^{\circ}\otimes H^1(X,\Q)\big) \to J\big(
H^1(\ol{C}_1,\Q(1))\otimes H^1(\ol{C}_2,\Q(1))\otimes H^1(X,\Q)\big)\big].
\end{equation*}

\noindent
{\rm (ii)}  $\Gamma\big(H^3_E(\ol{S},\Q(2))^{\circ} \otimes H^{1}(X, \Q))\ \simeq$
\[
 \Gamma \biggl(\big[H^1(\ol{C}_1,\Q)\otimes
H^0_{\deg 0}(\Sigma_2,\Q) \bigoplus H^0_{\deg 0}(\Sigma_1,\Q)\otimes H^1(\ol{C}_2,\Q)\big]\otimes H^1(X,\Q(1))\biggr).
\]
\end{lemma}

\begin{proof}
Part (i): Observe that
\[
\frac{H^2(\ol{S},\Q)}{H^2_E(\ol{S},\Q)} = H^1(\ol{C}_1,\Q)\otimes H^1(\ol{C}_2,\Q),\quad
H^2(S,\Q) = H^1(C_1,\Q)\otimes H^1(C_2,\Q),
\]
as $H^2(C_j) = 0$. There is a short exact sequence:
\[
0\to \frac{H^2(\ol{S},\Q)}{H^2_E(\ol{S},\Q)} \to H^2(S,\Q) \to H^3_E(\ol{S},\Q)^{\circ} \to 0.
\]
This in turn gives rise to a short exact sequence:
\[
0\to H^1(\ol{C}_1,\Q(1))\otimes H^1(\ol{C}_2,\Q(1))\otimes H^1(X,\Q) \to
\]
\[
H^1({C}_1,\Q(1))\otimes H^1({C}_2,\Q(1))\otimes H^1(X,\Q) \to H^3_E(\ol{S},\Q(2))^{\circ}\otimes H^1(X,\Q) \to 0.
\]
So, using purity,
\begin{equation*}
\Gamma\big(H^1({C}_1,\Q(1))\otimes H^1({C}_2,\Q(1))\otimes H^1(X,\Q)\big) 
\end{equation*}
can be identified with 
\begin{equation*}
\ker \big[\Gamma\big(H^3_E(\ol{S},\Q(2))^{\circ}\otimes H^1(X,\Q)\big) \to J\big(
H^1(\ol{C}_1,\Q(1))\otimes H^1(\ol{C}_2,\Q(1))\otimes H^1(X,\Q)\big)\big].
\end{equation*}

Part (ii):  Poincar\'e duality gives an isomorphism of MHS
\begin{equation} \label{E3}
H^3_E(\ol{S},\Q(2)) \simeq H_1(E,\Q), \ {\rm hence}\ 
H^3_E(\ol{S},\Q)^{\circ} \simeq \ker \big(H_1(E,\Q) \to H_1(\ol{S},\Q)\big)
\,.
\end{equation}
Moreover the Mayer-Vietoris sequence gives  us the exact sequence
\[
0\to  \big[H_1(\ol{C}_1\times  \Sigma_2,\Q) \oplus H_1(\Sigma_1\times \ol{C}_2,\Q)\big]\otimes H^1(X,\Q) 
\]
\[
 \to H_1(E,\Q)\otimes H^1(X,\Q) \to H_0(\Sigma_1\times\Sigma_2,\Q)\otimes H^1(X,\Q)
\,.
\]
But $\Gamma \big(H_0(\Sigma_1\times\Sigma_2,\Q)\otimes H^1(X,\Q)\big) = 0$; moreover
one has a commutative diagram
\[
\begin{matrix}
H_1(\ol{C}_1\times \Sigma_2,\Q)\oplus H_1(\Sigma_1\times \ol{C}_2,\Q)&\hookrightarrow&H_1(E,\Q)\\
||&&\big\vert\\
H_1(\ol{C}_1,\Q)\otimes H_0(\Sigma_2,\Q)\oplus H_0(\Sigma_1,\Q)\otimes H_1(\ol{C}_2,\Q)&&\big\vert\\
\biggl\downarrow&&\biggr\downarrow\\
H_1(\ol{C}_1,\Q)\otimes H_0(\ol{C}_2,\Q) \oplus H_0(\ol{C}_1,\Q)\otimes H_1(\ol{C}_2,\Q)&\simeq&
H_1(\ol{S},\Q).
\end{matrix}
\]
Hence from this and~\eqref{E3},
$\Gamma \big(H^3_E(\ol{S},\Q(2))^{\circ}\otimes H^1(X,\Q)\big)$ can be identified
with
\[
 \Gamma \biggl(\big[H^1(\ol{C}_1,\Q)\otimes
H^0_{\deg 0}(\Sigma_2,\Q) \bigoplus H^0_{\deg 0}(\Sigma_1,\Q)\otimes H^1(\ol{C}_2,\Q)\big]\otimes H^1(X,\Q(1))\biggr).
\]
\end{proof}

\begin{proof}[Proof of Theorem~\ref{MT0}]
Note that by the Lefschetz (1,1) theorem, $\Gamma\big(H^1(\ol{C}_j,\Q)\otimes H^1(X,\Q)(1)\big)$ is
algebraic.  Let us assume for the moment that there exists
$B$ in $ \Gamma\big( H^0_{\deg 0}(\Sigma_1,\Q)\otimes H^1(\ol{C}_2,\Q)\otimes H^1(X,\Q(1))\big)$ of the
form $B = \xi\times D$, where $D \subset \ol{C}_2\times X$ is an irreducible curve, $\xi \in
H^0_{\deg 0}(\Sigma_1,\Q)$, and that $B$ is in the kernel of the Abel-Jacobi map
in Lemma~\ref{E2lemma}(i). Notice that the inclusion $H^1(\ol{C}_1,\Q(1))\otimes \Q(1)\cdot [D] \hookrightarrow
H^1(\ol{C}_1,\Q(1))\otimes H^1(\ol{C}_2,\Q(1))\otimes H^1(X,\Q)$ defines a splitting, and hence
an inclusion 
\[
J\big(H^1(\ol{C}_1,\Q(1))\big) \hookrightarrow J\big(
H^1(\ol{C}_1,\Q(1))\otimes H^1(\ol{C}_2,\Q(1))\otimes H^1(X,\Q)\big).
\]
By applying Abel's theorem to $\ol{C}_1$, it follows that there exists
$f\in\C( \ol{C}_1\times D)^{\times}$ for which $(f) = \xi\times D = B$, thus
supplying the necessary element in $\CH^2(S \times X, 1 ; \Q)$. The same story holds if we replace
$D$ by any divisor with non-trivial image in the Neron-Severi
group. Using a basis for the Neron-Severi group of $\ol{C}_2\times X$,
one sees that the kernel of the Abel-Jacobi map restricted to
$ \Gamma\big( H^0_{\deg 0}(\Sigma_1,\Q)\otimes H^1(\ol{C}_2,\Q)\otimes H^1(X,\Q(1))\big)$ 
is in the image of $ \cl_{1,2} $.
A similar story holds separately for
$A$ in $\Gamma\big( H^1(\ol{C}_1,\Q)\otimes H^0_{\deg 0}(\Sigma_2,\Q)\otimes H^1(X,\Q(1))\big)$
in the kernel of the Abel-Jacobi map.

The more complicated issue is the case where $ A+B $ in
\[
\Gamma \biggl(\big[H^1(\ol{C}_1,\Q)\otimes
H^0_{\deg 0}(\Sigma_2,\Q) \bigoplus H^0_{\deg 0}(\Sigma_1,\Q)\otimes H^1(\ol{C}_2,\Q)\big]\otimes H^1(X,\Q(1))\biggr)
\]
is in the kernel of the Abel-Jacobi map.  The problem boils down to the
following.  There are two subHodge structures
$V_1,\ V_2$ of $H^1(\ol{C}_1,\Q(1))\otimes H^1(\ol{C}_2,\Q(1))\otimes H^1(X,\Q)$, where
$V_1 = H^1(\ol{C}_1,\Q(1))\otimes \Gamma\big(H^1(\ol{C}_2,\Q(1))\otimes H^1(X,\Q)\big)$,
and $V_2 \simeq H^1(\ol{C}_2,\Q(1))\otimes \Gamma\big(H^1(\ol{C}_1,\Q(1))\otimes H^1(X,\Q)\big)$
is defined similarly. If their intersection $ V $ is trivial,
then
\[
J(V_1) \oplus J(V_2)\hookrightarrow J\big(H^1(\ol{C}_1,\Q(1))\otimes H^1(\ol{C}_2,\Q(1))\otimes H^1(X,\Q)\big),
\]
so $ A+B$ in the kernel of the Abel-Jacobi map implies that
$A,\ B$ are in the kernel, and from our earlier discussion it
follows that then
\[
\cl_{2,1} :\CH^2(S\times X,1;\Q) \to \Gamma H^3(S\times X,\Q(2))
\]
is surjective. 
If $ V $ is non-trivial, then from types we see that $ V(-2) $ is contained in
\[
 \big\{H^{1,0}(\ol{C}_1)\otimes H^{1,0}(\ol{C}_2) \otimes 
H^{0,1}(X)\big\} \bigoplus \big\{H^{0,1}(\ol{C}_1)\otimes H^{0,1}(\ol{C}_2) \otimes 
H^{1,0}(X)\big\}
\]
inside $ H^2(\ol{S}, \Q) \otimes H^1(X, \Q) $.
Tensoring this with $ H^{2d -1}(X, \Q(d))$, where $d = \dim X$,
and applying the cup product, 
\[
 H^1(X, \Q) \times H^{2d -1}(X, \Q(d)) \xrightarrow{\cup}
H^{2d}(X,\Q(d)),
\] 
followed by the identification
$H^{2d}(X,\Q(d)) \simeq \Q $, we find that $ V(-2) $ results in
a non-trivial $\Q$-subHodge structure 
of $ H^2(\ol{S} , \Q) $ contained in
$ H^{2,0}(\ol{S}) \bigoplus H^{0,2}(\ol{S}) $.
\end{proof}

We conclude this section with some discussion of the rigidity condition appearing in Theorem~\ref{MT0}.

\begin{ex} \label{CMex}
If the $\ol{C}_j$ are elliptic curves, then the existence of
a non-trivial (hence rank $2$) $\Q$-subHodge structure
implies that $\ol{C}_j = \C/\{\Z\oplus \Z\tau_j\}$, where $\Q(\tau_j)/\Q$ is a quadratic extension.
To see this, observe that this subHodge structure by Poincar\'e
duality corresponds to two independent classes in $ H_2(\ol{C}_1 \times \ol{C}_2, \Q) $,
\[
\xi_j = k_1^{(j)}\alpha_1\otimes \alpha_2 + k_2^{(j)}\alpha_1\otimes \beta_2 + 
k_3^{(j)}\beta_1\otimes \alpha_2 + k_4^{(j)}\beta_1\otimes \beta_2 \qquad (j=1,2),
\]
where $\alpha_j = \ol{[0,1]} $, $\beta_j = \ol{[0,\tau_j]}$,
and all $ k_i^{(j)} $ are in $ \Q $. Thus because of types we
have
\[
0 = \int_{\xi_j}dz_1\wedge d\ol{z}_2 = k^{(j)}_1 + k^{(j)}_2\ol{\tau}_2 + k^{(j)}_3\tau_1 + 
k^{(j)}_4\tau_1\ol{\tau}_2
\qquad (j=1,2)
,
\]
so that
\begin{equation*}
 \tau_1 = - \frac{k^{(j)}_2\ol{\tau}_2 + k^{(j)}_1}{k^{(j)}_4\ol{\tau}_2+k^{(j)}_3}
\qquad (j=1,2)
.
\end{equation*}
Using $j = 1,2$, we can solve for $\tau_2$, viz.,
\[
[k^{(1)}_2k^{(2)}_4 - k^{(2)}_2k^{(1)}_4]\tau_2^2 + \cdots = 0.
\]
But $\xi_1,\ \xi_2$ independent and Im$(\tau_j)\ne 0$ implies that
$k^{(1)}_2k^{(2)}_4 - k^{(2)}_2k^{(1)}_4 \ne 0$.
Then $ \Q(\tau_1) = \Q(\ol{\tau}_2) = \Q(\tau_2) $,
so that the $ \ol{C_j} $ are isogenous and have complex multiplication.
Therefore the Neron-Severi group of $ \ol{C}_1 \times \ol{C}_2  $
has rank 4, and gives rise to a subHodge structure
$ W \subseteq H^2(\ol{C}_1 \times \ol{C}_2, \Q(1)) $ of dimension~4,
with $ W_\C = H^{1,1}(\ol{C}_1 \times \ol{C}_2) $.
Its orthogonal $ V $ under the cup product 
$ H^2(\ol{C}_1 \times \ol{C}_2, \Q(1))  \otimes H^2(\ol{C}_1 \times \ol{C}_2, \Q(1)) \to
H^4(\ol{C}_1 \times \ol{C}_2, \Q(2)) $
is now a subHodge structure of $ H^2(\ol{C}_1 \times \ol{C}_2, \Q(1)) $
with
$ V_\C = H^{2,0}(\ol{C}_1 \times \ol{C}_2) \oplus H^{0,2}(\ol{C}_1 \times \ol{C}_2) $.
\end{ex}

Note that a simple $\ol{\Q}$-spread argument
implies that the  $\ol{C}_j$ in Example~\ref{CMex}
are defined over $\ol{\Q}$ (and
more precise statements are known classically). This leads to

\begin{q} Let $W$ be a smooth projective surface.
If $H^{2,0}(W)\oplus H^{0,2}(W)$ contains a non-trivial $\Q$-subHodge Structure of $H^2(W,\Q)$,
can $W$ be obtained by base extension from a surface defined over $\ol{\Q}$?
\end{q}

\begin{prop}
Suppose $X$ is a K3 surface or an abelian surface. Then the answer to the previous 
question is positive.
\end{prop}

\begin{proof}
In both cases, $H^{0,2}$ is one dimensional, so the assumption
implies that $H^{2,0}(W)\oplus H^{0,2}(W)$
arises from a $\Q$-subHodge structure of $H^2(W,\Q)$. Since the former is a Hodge structure of type
$(1,0,1)$, it follows that its Mumford-Tate group is abelian. In the case of an abelian surface, this implies that
the Mumford-Tate group of the abelian variety is abelian, and therefore the abelian surface has complex multiplication.
On the other hand, every CM abelian variety over $\C$ is defined over number field. If $W$ is a $K3$ surface,
then it has maximal Picard rank, hence is well-known to be rigid, {\em a fortiori,} defined over $\ol{\Q}$.
\end{proof}

\begin{rmk}In  spite of the mild ``rigidity'' assumption in Theorem~\ref{MT0},
any attempt to extend the theorem to the generic point of $S$, without
introducing some conjectural assumptions, seems incredibly
difficult.
\end{rmk}

\section{Generalities} \label{generalsection}

Before proceeding further, and to be able to move further ahead,
we explain some necessary assumptions.  
{\em For the remainder of this paper (with the occasional reminder), we assume the following:}

\begin{ass}\label{ASS}  (i) The Hodge conjecture.

\medskip
(ii) The Bloch-Beilinson conjecture on the injectivity of the Abel-Jacobi map for Chow groups 
of smooth projective varieties defined over $\ol{\Q}$ (see~\cite[Conj 3.1 and \S4]{Lew2}).
\end{ass}

To spare the reader of time consuming search of  multiple  sources by many others, we  refer mostly
to~\cite{Lew2}, for all the necessary statements and details.
Let $Z/\C$ be any smooth projective variety of dimension $d_0$, $r\geq 0$, and 
put 
\[
\CH_{AJ}^r(Z;\Q)  = \ker \big(AJ : \CH^r_{\hom}(Z;\Q) \to J\big(H^{2r-1}(Z,\Q(r))\big).
\]
As mentioned in \S\ref{S1}, we  recall the notion of a descending Bloch-Beilinson (BB) filtration
$\{F^{\nu}\CH^r(Z;\Q)\}_{\nu\geq0}$, with $F^{0}\CH^r(Z;\Q) = \CH^r(Z;\Q)$, $F^{1}\CH^r(Z;\Q) = \CH^r_{\hom}(Z;\Q)$,
$F^{r+1}\CH^r(Z;\Q) = 0$, and satisfying a number of properties codified
 for example in ~\cite[\S11]{Ja2}, \cite[\S4]{Lew2}. There is also  the explicit construction of
a candidate BB filtration by Murre, based on a conjectured Chow-K\"unneth decomposition, 
and subsequent conjectures in~\cite{M}, 
which is equivalent to the existence
of a BB filtration  as formulated in~\cite[\S11]{Ja2}).  Further, Jannsen also proved that the BB filtration is unique if
it exists.  The construction of the filtration in~\cite{Lew1} (and used in~\cite{Lew2})
relies on Assumptions~\ref{ASS}, which if
true, provides the existence of a BB filtration, and hence is
the same filtration as Murre's, by the aforementioned uniqueness.

All candidate filtrations seem to show that $F^2\CH^r(Z;\Q)\subseteq \CH_{AJ}^r(Z;\Q)$. The
following is considered highly non-trivial:

\begin{conj} \label{C1}
{\rm $\CH_{AJ}^r(Z;\Q) = F^2\CH^r(Z;\Q)$.}
\end{conj}
In light  of  Assumptions~\ref{ASS}, this is equivalent to the
surjectvity of
\[
\cl_{1,r} : \CH^r(\Spec(\C(Z)) , 1;\Q) \twoheadrightarrow \Gamma\big(H^{2r-1}(\C(Z),\Q(r))\big)
\]
as in~\cite[(S3)]{dJ-L},  provided both statements apply to all smooth projective $Z/\C$.
 For a proof, see~\cite[Thm~1.1]{Lew2}.
 
 \bigskip

 One of the key properties of the BB filtration is the factorization of graded pieces of that
 filtration through the Grothendieck motive. Let $\Delta_Z$ in $ \CH^{d_0}(Z\times Z)$ be the diagonal class,
 with cohomology class $[\Delta_Z]$ in $ H^{2d_0}(Z\times Z,\Z(d_0))$.
Write 
$ \Delta_Z = \sum_{p+q=2d_0} \Delta_Z(p,q) $ in $ \CH^{d_0}(Z \times Z; \Q)$
such that
\[
 \bigoplus_{p+q=2d_0}[\Delta_Z(p,q)]\in \bigoplus_{p+q=2d_0}H^p(Z,\Q)\otimes H^q(Z,\Q)(d_0) = H^{2d_0}(Z\times Z,\Q(d_0)).
 \]
is the K\"unneth decomposition of $ [\Delta_Z] $.
 Then
 \[
 \Delta_Z(p,q)_*\big|_{Gr_F^{\nu}\CH^r(Z;\Q)}
 \]
is independent of the choice of $\Delta_Z(p,q)$.
(This is essentially due to the fact that $F^1\CH^r(Z;\Q) = \CH^r_{\hom}(Z;\Q)$, and functoriality properties of the
 BB filtration.) Furthermore,
 \begin{equation} \label{Delta1}
   \Delta_Z(2d_0-2r+\ell,2r-\ell)_*\big|_{Gr_F^{\nu}\CH^r(Z;\Q)}
=
  \delta_{\ell,\nu} \textup{id}_{Gr_F^{\nu}\CH^r(Z;\Q)}
 \end{equation}
with $ \delta_{i,j} $ the Kronecker delta.
  Consequently,
\begin{equation} \label{Delta2}
  \Delta_Z(2d_0-2r+\nu,2r-\nu)_*\CH^r(X;\Q) \simeq Gr_F^{\nu}\CH^r(X;\Q),
\end{equation}
  and accordingly there is a non-canonical decomposition
  \[
  \CH^r(Z;\Q) = \bigoplus_{\nu=0}^r\Delta_Z(2d_0-2r+\nu,2r-\nu)_*\CH^r(Z;\Q).
  \]
In summary, we will view the kernel of the Abel-Jacobi map in terms
of $Gr_F^{\nu}\CH^r(Z;\Q)$ for $\nu \geq 2$, i.e., under Conjecture~\ref{C1},
\begin{equation*}
\CH^r_{AJ}(Z;\Q) = F^2\CH^r(Z;\Q)  \simeq \bigoplus_{\nu=2}^rGr_F^{\nu}\CH^r(Z;\Q),
\end{equation*}
(to re-iterate, non-canonically). 

\begin{rmk}\label{RM}
Murre's idea~\cite{M} is that one can choose
the $ \Delta_Z(p,q) $ to be  commuting, pairwise orthogonal idempotents.
 By Beilinson and Jannsen, such a lift is possible if $\CH_{\hom}^{d_0}(Z\times Z;\Q)$ 
 is a nilpotent ideal under composition, which is a consequence of Assumptions \ref{ASS}.
 It should be pointed out that such $\Delta_Z(p,q)$'s are still not unique.
\end{rmk}

The following will play an important role in Section~\ref{SMR}.

\begin{prop}\label{GP}
Under Assumptions~\ref{ASS}, the map
\[
\Xi_{Z,\ast} := \bigoplus_{\nu=2}^r\Delta_Z(2d_0-2r+\nu,2r-\nu)_*: F^2\CH^r(Z;\Q) \to F^2\CH^r(Z;\Q),
\]
is an isomorphism. Moreover, if the $\Delta_Z(p,q)$'s are chosen as in Murre's Chow-K\"unneth
decomposition (viz., in Remark \ref{RM}), then it is the identity.
\end{prop}

\begin{proof} If $r=2$, this is obvious, as $F^r\CH^r(Z;\Q) = Gr_F^r\CH^r(Z;\Q)$. So assume
$r>2$. The $5$-lemma, together with the diagram
\[
\begin{matrix}
&&F^r\CH^r(Z;\Q)\\
&\\
&&||\\
&\\
0&\to&Gr_F^r\CH^r(Z;\Q)&\to&F^{r-1}\CH^r(Z;\Q)&\to&Gr_F^{r-1}\CH^r(Z;\Q)&\to&0\\
&\\
&&\Xi_{Z,\ast}\biggl\downarrow ||\quad&&\Xi_{Z,\ast}\biggl\downarrow\quad
&&\Xi_{Z,\ast}\biggl\downarrow ||\quad\\
&\\
0&\to&Gr_F^r\CH^r(Z;\Q)&\to&F^{r-1}\CH^r(Z;\Q)&\to&Gr_F^{r-1}\CH^r(Z;\Q)&\to&0
\end{matrix}
\]
tells us that the middle vertical arrow is an isomorphism. By an inductive-recursive argument, we arrive at
another $5$-lemma argument:
\[
\begin{matrix}
0&\to&F^3\CH^r(Z;\Q)&\to&F^{2}\CH^r(Z;\Q)&\to&Gr_F^{2}\CH^r(Z;\Q)&\to&0\\
&\\
&&\Xi_{Z,\ast}\biggl\downarrow\wr\quad&&\Xi_{Z,\ast}\biggl\downarrow\quad
&&\Xi_{Z,\ast}\biggl\downarrow ||\quad\\
&\\
0&\to&F^3\CH^r(Z;\Q)&\to&F^{2}\CH^r(Z;\Q)&\to&Gr_F^{2}\CH^r(Z;\Q)&\to&0
\end{matrix}
\]
which implies the isomorphism in the proposition.

\medskip

For the second statement, clearly
$ \Delta_Z = \oplus_{\ell=2r-2d_0}^{2r} \Delta_Z(2d_0-2r+\ell,2r-\ell) $
induces the identity.
But from~\eqref{Delta2} we see that the terms with $ \ell \ge r+1 $
do not contribute because $ Gr_F^\ell \CH^r(Z;\Q) = 0 $. Also, from~\eqref{Delta1} we find that if
$ \ell \ne \nu $, then $ \Delta_Z(2d_0-2r+\ell,2r-\ell) $ maps
$ F^\nu \CH^r(Z;\Q) $ into $ F^{\nu+1} \CH^r(Z;\Q) $.
But if $ \ell < \nu $ and this correspondence is  idempotent then we can
iterate this, finding it kills $ F^\nu \CH^r(Z;\Q) $
as $ F^{r+1} \CH^r(Z;\Q) = 0 $. 
Taking $ \nu =2 $, we are done.
\end{proof}

So to understand
more about $\CH^r_{AJ}(Z;\Q)$, it makes sense to study
\[ 
\Delta_Z(2d_0-2r+\nu,2r-\nu)_* \CH^r_{AJ}(Z;\Q),
\]
for $ 2 \le \nu \le r$.

\medskip

\section{Main results}\label{SMR}

In this section, we will be assuming Assumptions~\ref{ASS} as well
as Conjecture~\ref{C1}.
{\em We shall also assume that we are working with Murre's Chow-K\"unneth decomposition.}
Furthermore, $ \ol{S} $ and $ X $ are assumed smooth and projective,
with $\dim \ol{S} = N$, $\dim X = d$. Then we have the results
of Section~\ref{generalsection} for $ Z = \ol{S} \times X $,
with dimension $ d_0 = N + d $.

\begin{rmk}\label{KR}
This remark is critical to understanding our approach to
the main results in the remainder of this paper. We want to take
some earlier results, in particular Proposition~\ref{GP}, one step further.
Let us assume the notation and setting in Proposition~\ref{KeyP},
for $\ol{\X} = \ol{S}\times X$, with~$\K$ as in~\eqref{eq:2.5}.
Our goal is to show  that the RHS of  Proposition~\ref{KeyP} is zero; in particular,
that 
\[
\Xi_0 := \ker\big(AJ : \K\to J\big(H^{2r-1}(\ol{S}\times X,\Q(r))\big)\big)
\]
is contained in $ N^1_{\ol{S}}\CH^r(\ol{S}\times X;\Q(r)) $.
Proposition~\ref{GP} shows that the induced map
\begin{equation} \label{EE3}
\bigoplus_{\nu=2}^r\Delta_{\ol{S}\times X}(2(N+d)-2r+\nu,2r-\nu)_*: \Xi_0\to \Xi_0,
\end{equation}
is  the identity. 
\end{rmk}

Below we only consider $\ell\leq N$ since
ultimately we will be passing to the generic point $\eta_{\ol{S}}$ of $\ol{S}$, hence through an
affine $S\subset \ol{S}$, where we apply the affine weak Lefschetz theorem, with $D =\ol{S}\bs S$.
Although not needed, we could also take $\ell \geq 2$ because the cases $\ell = 0, 1$
can be proved as in the proof  of Proposition~\ref{splitcase1}(2).

\medskip
Note that
\begin{equation} \label{E0E}
\Gamma\big(H^{2r-1}(S\times X,\Q(r))\big) = \bigoplus_{\ell=1}^N\Gamma\biggl(H^{\ell}(S,\Q(1))\otimes H^{2r-\ell-1}(X,\Q(r-1))\biggr),
\end{equation}
hence we break down our arguments involving each  of the $N$ terms on the RHS of~\eqref{E0E}.
This is similar to how we handled things in \S\ref{S2}.
Note that if we apply the K\"unneth projector
$[\Delta_{\ol{S}}\otimes \Delta_X(2d-2r+\ell+1,2r-\ell-1)]_*$
to the short exact sequence in~\eqref{eq:2.2} of \S\ref{SN} (with $\Y = D\times X$),
where the action of the aforementioned K\"unneth projector on $H^{2r-1}(S\times X,\Q(r))$
is given by ${\rm Pr}_{13,*} \big({\rm Pr}_{23}^*[\Delta_X(2d-2r+\ell+1,2r-\ell-1)]\cup {\rm Pr}_{12}^*(-)\big)$, 
observing that both ${\rm Pr}_{13},\ {\rm Pr}_{12}:
S\times X\times X \to S\times X$ are proper and flat, we end up with the
short exact sequence:
\begin{equation*}
0 \to \biggl\{\frac{H^{\ell}(\ol{S},\Q(1))}{{H}^{\ell}_D(\ol{S},\Q(1))}\biggr\}\otimes H^{2r-\ell-1}(X ,\Q(r-1))\to H^{\ell}(S,\Q(1))
\otimes H^{2r-\ell-1}(X ,\Q(r-1))
\end{equation*}
\[
 \to  H^{\ell+1}_D(\ol{S},\Q(1))^{\circ}\otimes H^{2r-\ell-1}(X ,\Q(r-1))\to 0,
 \]
 where 
\[
H^{\ell+1}_D(\ol{S},\Q)^{\circ} = \ker\big(H^{\ell+1}_D(\ol{S},\Q) \to H^{\ell+1}(\ol{S},\Q)\big).
\]
This accordingly modifies the bottom row of~\eqref{eq:2.3} of \S\ref{SN}  in the obvious way.
As a reminder, the K\"unneth components cycle representatives $\Delta_X(2d-\bullet,\bullet)$ 
of the diagonal class $\Delta_X$  (and of $\Delta_{\ol{S}}$, hence the product
$\Delta_{\ol{S}\times X} = \Delta_{\ol{S}}\otimes \Delta_X$)
are now assumed chosen in the sense of Murre (see Remark~\ref{RM}).
So we can likewise apply the projector $\Delta_{\ol{S}}\otimes \Delta_X(2d-2r+\ell+1,2r-\ell-1)$
to the top row of~\eqref{eq:2.3} of \S\ref{SN}, and arrive at a modified commutative
diagram~\eqref{eq:2.3}, based on $\ell = 1,...,N$. We can be more explicit here.
For the sake of brevity, 
let us denote $ \Delta_X(2d-2r+\ell+1,2r-\ell-1) $ in $ \CH^d(X\times X;\Q)$ here with
$ P $, and let $\Y = D\times X$ for some codimension one
subscheme  $D\subset \ol{S}$. Then $\Delta_{\ol{S}}\otimes P$ acts on
$\CH^r(\ol{S}\times X,1;\Q)$ in a natural way, and its action on 
$\CH^r_{D\times X}(\ol{S}\times X;\Q)$ is given as follows.
First of all, $\CH^r_{D\times X}(\ol{S}\times X;\Q) = \CH^{r-1}(D\times X;\Q)$. There are proper flat maps
$D\times X \times X$ given by projections $Pr_{12}: D\times X \times X \to D\times X$, $Pr_{23}: D\times X \times X\to X\times X$
and $Pr_{13}:D\times X\times X\to D\times X$.
For $\gamma$ in $\CH^{r-1}(D\times X;\Q)$, $Pr_{12}^*(\gamma)$
in $ \CH^{r-1}(D\times X\times X;\Q)$ is defined (flat pullback). 
For $P\in \CH^d(X\times X;\Q)$, the intersection 
$Pr_{12}^*(\gamma) \bullet Pr_{23}^*(P)\in \CH^{r-1+d}(D\times X\times X;\Q)$
is likewise well defined~\cite[\S2]{F}. The action then is given by
$Pr_{13,*}\big(Pr_{12}^*(\gamma) \bullet Pr_{23}^*(P)\big)$ in
$ \CH^{r-1}(D\times X;\Q) = \CH^r_{D\times X}(\ol{S}\times X;\Q)$.
The action of $\Delta_{\ol{S}}\otimes P$ on $\CH^r_{\hom}(\ol{S}\times X;\Q)$
is clear. Finally, by an elementary Hodge theory argument, one arrives
at a modified version of Proposition~\ref{KeyP}. Specifically, $\Delta_{\ol{S}}\otimes P$
acts naturally on all terms on the RHS of the display in Proposition~\ref{KeyP}, making
use of  functoriality of the Abel-Jacobi map, and operates naturally on the LHS,
as clearly evident in the above discussion following (\ref{E0E}). As $\ell$ ranges from $1,...,N$, both sides
of the aforementioned display decompose accordingly into a direct sum.

\medskip

We need to determine what 
$\Delta_{\ol{S}}\otimes \Delta_X(2d-2r+\ell+1,2r-\ell-1)$ does to $\gamma$, which is an
algebraic cycle of codimension $r$ (dimension $N+d-r$) on $\ol{S}\times X$, but supported on $D\times X$.
Decomposing $ \Delta_{\ol{S}} $, we  can write 
\[
\Delta_{\ol{S}\times X}(2(N+d)-2r+\nu, 2r-\nu),
\]
as a sum of
\begin{equation} \label{OC}
\Delta_{\ol{S}}(2N+\nu-\ell-1,\ell+1-\nu)\otimes \Delta_X(2d-2r+\ell+1,2r-\ell-1)
.
\end{equation}
We recall that we have $  2 \le \nu \le r $, and, as indicated earlier, only consider $ \ell $
with $ 2 \le \ell \le N $.
Before stating our next result, it is helpful to introduce the following, which includes those $\ol{S}$ 
which are complete intersections in projective space or more generally a Grassmannian.

\begin{lemma} \label{oddevenlemma}
Suppose that $\ol{S}$ is a  variety of dimension
$ N $ such that for $i\ne N$:
\[
H^i(\ol{S},\Q)  \ is \ zero\ for \ i\ odd \ and \ generated\ by \  algebraic\ cycles\ for\ i\ even.
\]
Then $\ol{S}$ admits a Chow-K\"unneth decomposition in the sense of~\cite{M},
Remark~\ref{RM} with the supports
of the K\"unneth projectors compatible with the supports of the cohomology classes
in $H^{\bullet}(\ol{S},\Q)$. Specifically, for $j\ne N$,
$ \Delta_{\ol{S}}(2N-2j,2j) $ is contained in the image of 
$ \CH^{N-j}(\ol{S}; \Q) \otimes \CH^j(\ol{S}; \Q) $
under pullback to $ \CH^r(\ol{S} \times \ol{S} ; \Q) $ and taking the product..
\end{lemma}

\begin{proof}
For $ j = 0,\dots, N $, 
let $W_{2N-2j}$ in $ \CH^{N-j}(\ol{S},\Q)$ and $V_{2j}$
in $ \CH^j(\ol{S},\Q)$ be algebraic cycles
such that $\lambda_j := \deg\big(\langle W_{2N-2j},V_{2j}\rangle_{\ol{S}}\big)  \ne 0$. Then
\[
\big\{ W_{2N-2k}\times V_{2k}\big\}\circ\big\{W_{2N-2j}\times V_{2j}\big\}
= \begin{cases} 0&\text{if $k\ne j$}\\
\lambda_j\cdot \{W_{2N-2j}\times V_{2j}\}&\text{if $k=j$}
\end{cases}
\]
To see this, compute
$ \big\{ W_{2N-2k}\times V_{2k}\big\}\circ\big\{W_{2N-2j}\times V_{2j}\big\} $
as
\begin{equation} \label{PReq}
Pr_{13,*}\big(\langle Pr_{12}^*(W_{2N-2k}\times V_{2k}),Pr_{23}^*(W_{2N-2j}\times V_{2j})\rangle_{\ol{S}\times \ol{S}\times\ol{S}}\big)
\end{equation}
in $\CH^N(\ol{S}\times\ol{S}; \Q)$.
If $ j=k $ the statement is clear. If
 $k>j$ then $V_{2k}\cap W_{2N-2j}=0$ by codimension, 
and if $ k < j $ then this intersection has dimension at least
1, so~\eqref{PReq}
has codimension bigger than~$ N$. In either case it is trivial.
This  principle allows us to define, for $ i \ne N $,
mutually orthogonal idempotents
$\pi_i$ in $ \CH^N(\ol{S}\times\ol{S};\Q)$ with
$[\pi_i] $ in $ H^{2N-i}(\ol{S},\Q)\otimes H^i(\ol{S},\Q)(N)$.
We can take $\pi_i = 0$ if $i\ne N$ odd,
so that $ [\pi_i] = [\Delta(2N-i,i)] = 0 $ by our assumption
on the odd cohomology groups.
For even $i=2j\ne N$, we can arrange that $[\pi_{2j}] = [\Delta(2N-2j,2j)]$ 
by using the assumption about the even cohomology groups being
generated by algebraic cocyles, as well as the
non-degeneracy of the intersection product.
Put $ \pi_N = \Delta_{\ol{S}} - \sum_{i \ne N} \pi_i $.
Because the $ \pi_i $ for $ i \ne N $ are mutually orthogonal
idempotents by construction, the same holds if we use all $ \pi_0, \dots, \pi_{2N} $.
Because $ [\pi_i] = [\Delta_{\ol{S}}(2N-i, i)] $ for $ i \ne N $,
it follows from the definition of $\pi_N$ that $ [\pi_N] = [\Delta_{\ol{S}}(N,N)] $ as well.
\end{proof}

\begin{thm} \label{MT1}
Suppose that $\ol{S}$ is a  variety of dimension
$ N $ such that for $i\ne N$:
\[
H^i(\ol{S},\Q)  \ is \ zero\ for \ i\ odd \ and \ generated\ by \  algebraic\ cycles\ for\ i\ even.
\]
Given  Assumptions~\ref{ASS} and Conjecture~\ref{C1},
then
\[
\CH^r(X_{\eta_{\ol{S}}},1;\Q) \to \Gamma\big(H^{2r-1}(X_{\eta_{\ol{S}}},\Q(r))\big),
\]
is surjective.
\end{thm}

\begin{proof}
We shall use the Chow-K\"unneth decomposition of $ \Delta_{\ol{S}} $
as in Lemma~\ref{oddevenlemma}.
Observe that with regard to~\eqref{OC},
\[
\Delta_{\ol{S}}(2N+\nu-\ell-1,\ell+1-\nu) = \Delta_{\ol{S}}(N,N) \Leftrightarrow \ell+1-\nu = N.
\]
But $\ell \leq N$  and $\nu \geq 2$, so this never happens. Also,
the situation where $\ell+1-\nu=0$ does not contribute.
Namely, remember that
 $ \gamma \in \CH^r_{D\times X}(\ol{S}\times X;\Q)^{\circ}$ maps to
 a class in $\CH_{AJ}^r(\ol{S}\times X;\Q)$. Since
$|\gamma|\subset D\times X$,
$\Delta_{\ol{S}}(2N,0) = \{p\}\times \ol{S}$ for some $p\in \ol{S}$, so that $D\times \ol{S}$ doesn't
meet  $\Delta_{\ol{S}}(2N,0)$ for a suitable choice of $p$, it follows in this case  that
\[
\big(\Delta_{\ol{S}}(2N,0)\otimes \Delta_X(2d-2r+\nu,2r-\nu)\big)_*(\gamma) = 0.
\]
For $ 1 \le \ell+1-\nu \le N-1 $, and 
$\gamma$ in $\K$, we have that
\[
\big(\Delta_{\ol{S}}(2N+\nu-\ell-1,\ell+1-\nu)\otimes \Delta_X(2d-2r+\ell+1,2r-\ell-1)\big)_*(\gamma)
\]
is in $ N^1_{\ol{S}}\CH^r(\ol{S}\times X;\Q) $.
This is immediate from the fact that $\gamma$ is null-homologous
on $ \ol{S} \times X $
and the support of the Chow-K\"unneth components here,
as described in Lemma~\ref{oddevenlemma}.
Hence by Proposition~\ref{KeyP} (more precisely, the incarnation of Proposition~\ref{KeyP} in the
discussion following~\eqref{E0E}),
and~\eqref{EE3} of Remark~\ref{KR},
\[
\CH^r(X_{\eta_{\ol{S}}},1;\Q) \to \Gamma\big(H^{2r-1}(X_{\eta_{\ol{S}}},\Q(r))\big),
\]
is surjective.
\end{proof}

 \subsection{Grand finale}

For our final main result, we again consider
$\ol{S} = \ol{C}_1\times\cdots\times \ol{C}_N$, a product of smooth complete curves
(cf.\ Section~\ref{twocurve}).
As before, we restrict ourselves to  $ 2 \le \ell\leq N$.
Let us write
\[
\Delta_{\ol{C}_j} = e_j\times \ol{C}_j + \Delta_{\ol{C}_j}(1,1) + \ol{C}_j\times e_j,
\]
where $e_j\in \ol{C}_j$ and $ \Delta_{\ol{C}_j}(1,1)$ is defined by the equality. Consider the decomposition
\begin{equation} \label{ED}
\Delta_{\ol{S}} = \Delta_{\ol{C}_1}\otimes\cdots\otimes \Delta_{\ol{C}_N} = \bigotimes_{j=1}^N
\big\{e_j\times \ol{C}_j + \Delta_{\ol{C}_j}(1,1) + \ol{C}_j\times e_j\big\}.
\end{equation}
Note that $\Delta_{\ol{S}}(D) = D$. It needs to be determined what the RHS of~\eqref{ED}
 does to $D$, and more
precisely, what $\Delta_{\ol{S}}\otimes \Delta_X(2d-2r+\ell+1,2r-\ell-1)$ does to $\gamma$, which is an
algebraic cycle of dimension $N+d-r$ supported on $D\times X$. Now up to permutation, the RHS
of~\eqref{ED} is made up of terms of the form
\begin{equation} \label{summand}
( \Delta_{\ol{C}_j}(1,1))^{\otimes_{j=1}^{k_1}} \otimes
(\{e_{j}\times \ol{C}_j\})^{\otimes_{j=k_1+1}^{k_1+k_2}} \otimes
( \{\ol{C}_j\times e_j\}) ^ {\otimes_{j=k_1+k_2+1}^N}
,
\end{equation}
which is in the $ (k_1+2k_2, 2N-k_1-2k_2) $-component of $ \Delta_{\ol{S}} $.
Because in~\eqref{OC} we want 
$ \Delta_{\ol{S}}(2N+\nu-\ell-1,\ell+1-\nu) $, we have
\begin{equation} \label{E9E}
2N+\nu -\ell -1 =  k_1 + 2k_2
.
\end{equation}
Clearly, we have $ 0 \le k_1 + k_2 \le N $, $  2 \le \nu \le r $,
and we are restricting ourselves to $ 2 \le \ell \le N $.
Notice that if $k_1+k_2 < N$, we arrive at the situation where a correspondence in~\eqref{ED},
which when tensored with $\Delta_X$,
takes $\gamma$ to an element of $N^1_{\ol{S}}\CH^r(\ol{S}\times X;\Q)$.
If $k_1+k_2=N$, then  from~\eqref{E9E}, $N+\nu-\ell -1 = k_2\leq N$, and hence $\nu\leq \ell +1$. As in the
proof of Theorem~\ref{MT1}, we can ignore the case $\nu = \ell+1$.
 
 \begin{thm}\label{MT2}
Under Assumption~\ref{ASS}, and Conjecture~\ref{C1}, if $\ol{S} = \ol{C}_1\times\cdots \times \ol{C}_N$ is a product of smooth complete curves and $X$ a smooth projective variety, then for any $r\geq 1$,
\[
\CH^r(X_{\eta_{\ol{S}}},1;\Q) \to \Gamma\big(H^{2r-1}(X_{\eta_{\ol{S}}},\Q(r))\big),
\]
is surjective.
\end{thm}

 \begin{proof} We will prove this by induction on $N\geq 1$, the case $N=1$
 being part of Proposition~\ref{splitcase1}(2).
It will be crucial that no part of $\Delta_{\ol{S}}(N,N)$ occurs
because $ \ell \le N $ and $ \nu \ge 2 $ imply $ \ell + 1 - \nu < N $.
We shall argue on the summands of $ \Delta_{\ol{S}} $ that, up
to a permutation, are as in~\eqref{summand}.
The reductions preceding the theorem allow us to assume $k_1+k_2=N$, and
that $ 1\leq k_2  \leq N-1 $.
We see~\eqref{summand} is
 \begin{equation*}
\Xi :=  \Delta_{\ol{C}_1}(1,1)\otimes\cdots\otimes \Delta_{\ol{C}_{k_1}}(1,1)\otimes \{e_{k_1+1}\times
 \ol{C}_{k_1+1}\}\otimes\cdots\otimes \{e_N\times \ol{C}_N\}
 \end{equation*}
which is in the $ (N+k_2, N-k_2) $-component of $ \Delta_{\ol{S}} $.
Now let $ D \subset \ol{S} $ have codimension~1.
By choosing $\{e_{k_1+1},...,e_N\}$ appropriately,
 we can assume that
 \begin{equation*}
D' := \big| \Xi[D]\big|  {\subseteq} E \times \ol{C}_{k_1+1}\times\cdots\times\ol{C}_N
 \end{equation*}
 where $E \subset \ol{C}_1\times\cdots \times \ol{C}_{k_1}$ has codimension $1$.
Let $\gamma \in \CH^r_{AJ}(\ol{S}\times X;\Q) = F^2\CH^r(\ol{S}\times X;\Q)$, supported on $D\times X$, represent a class
\[
[\gamma]\in \Gamma\big(H^{\ell+1}_D(\ol{S},\Q(1))^{\circ}\otimes H^{2r-\ell-1}(X,\Q(r-1))\big),
\]
where
\[
H^{\ell+1}_D(\ol{S},\Q(1))^{\circ} = \ker\big(H^{\ell+1}_D(\ol{S},\Q(1)) \to H^{\ell+1}(\ol{S},\Q(1))\big).
\]
Then taking note of~\eqref{OC},
together with functoriality of the Abel-Jacobi map,
\[
\gamma' := \Xi_*(\gamma) \in \CH^r_{AJ}(\ol{S}\times X;\Q),
\]
is supported on $D'\times X$. Indeed, we have
\[
[\gamma']\in \Gamma\big(H^{\ell+1}_{D'}(\ol{S},\Q(1))^{\circ}\otimes H^{2r-\ell-1}(X,\Q(r-1))\big).
\]
By the properties of the BB filtration, and in light of Remark~\ref{KR}, 
we can reduce to the case where $\gamma = \gamma'$ and
$D = D'$.  Notice that
\[
H^{\ell+1}_{D'}(\ol{S},\Q)^{\circ} = \bigoplus_{j=1}^{2k_1-1}H^{j+1}_E(\ol{C}_1\times\cdots\times\ol{C}_{k_1},\Q)^{\circ}
\otimes H^{\ell-j}(\ol{C}_{k_1+1}\times\cdots\times \ol{C}_N,\Q).
\]
Thus we can reduce to $[\gamma']$ in
\[
\Gamma\big(H^{j+1}_E(\ol{C}_1\times\cdots\times\ol{C}_{k_1},\Q(1))^{\circ}
\otimes H^{\ell-j}(\ol{C}_{k_1+1}\times\cdots\times \ol{C}_N,\Q)\otimes H^{2r-\ell-1}(X,\Q(r-1))\big).
\]
Now let's put $\ol{S}_0 = \ol{C}_1\times\cdots\times\ol{C}_{k_1}$ and $X_0 = \ol{C}_{k_1+1}\times\cdots\times \ol{C}_N\times X$.
Then $\ol{S}_0\times X_0 = \ol{S}\times X$, and
$\gamma'\in \CH_{AJ}^r(\ol{S}_0\times X_0;\Q)$ is supported on $E\times X_0$. Further, $\dim \ol{S}_0 < \dim \ol{S}$.
Then by Proposition~\ref{KeyP}, and induction on $N$, 
\[
\gamma'\in N^1_{\ol{S}_0}\CH^r(\ol{S}_0\times X_0;\Q)\subset N^1_{\ol{S}}\CH^r(\ol{S}\times X;\Q).
\]
As mentioned above, the same applies to $\gamma$, and we are done.  
\end{proof}

\end{document}